\documentclass[12pt,reqno]{amsart}

\usepackage[pagewise]{lineno}

\usepackage{amsmath,amssymb,mathtools}
\usepackage[margin=1.1in]{geometry}
\usepackage{newtxtext,newtxmath}
\usepackage{bm}
\usepackage{microtype}
\usepackage{enumitem}
\usepackage{graphicx}
\usepackage{float}
\usepackage{booktabs}
\usepackage{xurl}
\usepackage{xcolor}
\usepackage{hyperref}
\usepackage[capitalise,nameinlink,noabbrev]{cleveref}

\definecolor{internalblue}{RGB}{0,0,0}
\definecolor{urlblue}{RGB}{0,0,0}
\hypersetup{colorlinks=true,linkcolor=internalblue,urlcolor=urlblue,citecolor=internalblue,pdfborder={0 0 0}}

\allowdisplaybreaks[2]

\makeatletter
\renewcommand{\subsection}{\@startsection{subsection}{2}%
  \z@{.5\linespacing\@plus.7\linespacing}{-.5em}%
  {\normalfont}}
\makeatother

\numberwithin{equation}{section}
\theoremstyle{plain}
\newtheorem{theorem}{Theorem}[section]
\newtheorem{lemma}[theorem]{Lemma}
\newtheorem{proposition}[theorem]{Proposition}
\newtheorem{corollary}[theorem]{Corollary}
\theoremstyle{definition}

\theoremstyle{remark}
\newtheorem{remark}[theorem]{Remark}

\crefname{equation}{equation}{equations}
\Crefname{equation}{Equation}{Equations}
\crefname{theorem}{theorem}{theorems}
\Crefname{theorem}{Theorem}{Theorems}
\crefname{lemma}{lemma}{lemmas}
\Crefname{lemma}{Lemma}{Lemmas}
\crefname{proposition}{proposition}{propositions}
\Crefname{proposition}{Proposition}{Propositions}
\crefname{corollary}{corollary}{corollaries}
\Crefname{corollary}{Corollary}{Corollaries}
\crefname{section}{section}{sections}
\Crefname{section}{Section}{Sections}
\crefname{figure}{figure}{figures}
\Crefname{figure}{Figure}{Figures}
\crefname{table}{table}{tables}
\Crefname{table}{Table}{Tables}

\graphicspath{{figures/}{./}}

\newcommand{\R}{\mathbb{R}}
\newcommand{\N}{\mathbb{N}}
\newcommand{\E}{\mathbb{E}}
\renewcommand{\P}{\mathbb{P}}

\DeclareMathOperator{\dist}{dist}
\DeclareMathOperator{\tr}{tr}
\DeclareMathOperator{\Var}{Var}

\DeclareMathOperator{\Vol}{Vol}
\newcommand{\dd}{\,d}

\newcommand{\inner}[2]{\left\langle #1,#2\right\rangle}
\newcommand{\refcite}[1]{%
  \allowbreak\hyperlink{cite.#1}{\mbox{\textcolor{internalblue}{\begin{NoHyper}\cite{#1}\end{NoHyper}}}}%
}
\newcommand{\refcitenote}[2]{%
  \allowbreak\hyperlink{cite.#2}{\mbox{\textcolor{internalblue}{\begin{NoHyper}\cite[#1]{#2}\end{NoHyper}}}}%
}

\newcommand{\midlabel}[2]{%
  \leavevmode
  \vadjust pre{%
    \vbox to 0pt{%
      \hbox to 0pt{%
        \pdfdest name {mid:#2} fitr
          width \linewidth
          height 1.5\baselineskip
          depth 1.5\baselineskip
        \hss}%
      \vss
    }%
  }%
  \label[#1]{#2}%
}
\newcommand{\iref}[1]{%
  \allowbreak\hyperlink{mid:#1}{\mbox{\textcolor{internalblue}{{\Cref*{#1}}}}}%
}

\title[Sharp Gaussian Asymptotics]
{Sharp Gaussian Asymptotics for Marginals of Euclidean Balls}

\author[B.S. Chen]{Bo-Si Chen}
\address{Bo-Si Chen \\ Department of Applied Mathematics\\National University of Tainan, 700301, Taiwan}
\email{bchen9572@gmail.com}

\author[Y.C. Huang]{Yen-Chang Huang}
\address{Yen-Chang Huang \\ Department of Applied Mathematics\\National Yang Ming Chiao Tung University, 30010, Taiwan}
\email{ychuang0802@nycu.edu.tw}

\thanks{Some of the results presented in this paper were obtained by the first author as an undergraduate through an NSTC College/University Student Research Project under the supervision of the second author. This work was supported by the National Science and Technology Council (NSTC), Taiwan, under Grant Nos.~114-2813-C-024-023-M and 115-2115-M-A49-001-MY2.}


\begin{document}
\begin{abstract}
We study Gaussian approximation for probability measures obtained by
normalizing one-dimensional profile functions, with one-dimensional
marginals of Euclidean balls as the principal example. We first establish
quantitative concentration estimates near the set of maximizers.
For profiles with a unique nondegenerate maximizer, we give sufficient
conditions under which centering at the maximizer and rescaling
according to the local quadratic approximation of the logarithm of
the profile yield densities that converge in $L^1(\R)$ to the
standard Gaussian density.

We then specialize to one-dimensional marginals of Euclidean balls
in $\R^n$. Writing $N=n-1$, we consider two standardizations of the
marginal distribution: one determined by the logarithmic curvature
at the maximizer, and the other by the exact standard deviation.
For each standardization, we identify the first-order correction,
of order $N^{-1}$, to the standard Gaussian density in $L^1(\R)$.
These expansions also determine the corresponding first-order
corrections to the probabilities of symmetric intervals and the
leading terms of the total variation distances from the standard
Gaussian distribution. In particular, under exact-variance
standardization, the total variation distance is asymptotic to an
explicit positive constant times $N^{-1}$. Consequently, the
previously known $O(N^{-1})$ bound for approximation by the Gaussian
distribution with the same variance is sharp in order.
\end{abstract}
\maketitle

\vspace{-1.5\baselineskip}


\section{Introduction}
\label{sec:introduction}

The volumes of parallel sections of a convex body are directly related to the
marginal densities of the uniform probability measure on the body. Let
$\lambda_k$ denote $k$-dimensional Lebesgue measure, and let $K\subset\R^n$
be a convex body with $\lambda_n(K)=1$. Let $X$ be uniformly distributed on
$K$. If $U\subset\R^n$ is an $m$-dimensional subspace, then the orthogonal
projection of $X$ onto $U$ has density
\[
\varphi_{K,U}(x)
:=
\lambda_{n-m}\bigl(K\cap(x+U^\perp)\bigr),
\qquad x\in U,
\]
where $\lambda_{n-m}$ denotes the induced Lebesgue measure on the affine
subspace $x+U^\perp$. When $m=1$, take $U=\operatorname{span}\{u\}$ with
$u\in S^{n-1}$. Then
\[
\varphi_{K,U}(tu)
=
\varphi_{K,u}(t)
=
\lambda_{n-1}
\bigl(K\cap\{z\in\R^n:z\cdot u=t\}\bigr).
\]
Thus the one-dimensional marginal density is the section profile of $K$ in
the direction $u$.

Gaussian approximation for such marginals has been studied in several
settings. For an isotropic random vector with a log-concave density, Eldan and
Klartag obtained pointwise Gaussian estimates for projections onto most
$m$-dimensional subspaces when $m\le n^c$, where $c>0$ is an absolute
constant \refcitenote{Theorem~1}{EldanKlartag2008}. Sodin obtained
tail-sensitive Gaussian estimates for average marginals and, under additional
hypotheses, for most one-dimensional marginals \refcite{Sodin2007}.

For the Euclidean ball $B_n\subset\R^n$ with $\lambda_n(B_n)=1$, Brehm and
Voigt \refcitenote{Theorem~3.1}{BrehmVoigt2000} proved that, for
$1\le m\le n-4$ and every $m$-dimensional subspace $U$,
\begin{equation}\label{eq:intro-brehm-voigt}
\int_U
\left|
\varphi_{B_n,U}(x)-g_{s_n^2,m}(x)
\right|\dd x
\le C\frac{m}{n},
\end{equation}
where
\[
g_{s^2,m}(x)
:=
\frac{1}{(2\pi s^2)^{m/2}}
\exp\left(-\frac{|x|^2}{2s^2}\right)
\]
is the centered $m$-dimensional Gaussian density with covariance
$s^2 I_m$, $s_n^2$ is the exact marginal variance, and $C$ is an absolute
constant. For $m=1$, this gives an $O(n^{-1})$ error in $L^1$. Since the total
variation distance between two measures with densities is one half of the
$L^1$ distance between their densities, it also gives an $O(n^{-1})$ bound in
total variation.

In this paper, we fix a direction and study the corresponding one-dimensional
marginal of a Euclidean ball. Write $N=n-1$. After rescaling the marginal
coordinate so that its support is $[-\pi/2,\pi/2]$, the profile becomes
\begin{equation}
\rho_n(\theta)
=
\left(1-\frac{4\theta^2}{\pi^2}\right)^{N/2},
\qquad -\frac{\pi}{2}\le\theta\le\frac{\pi}{2}.
\label{eq:intro-ball-profile}
\end{equation}
Let $\theta$ have density proportional to $\rho_n$, and set
$W_n=\frac12+\frac{\theta}{\pi}$. A change of variables gives
\[
W_n\sim\operatorname{Beta}\left(\frac{N+2}{2},\frac{N+2}{2}\right).
\]
Several results on Gaussian approximation for beta distributions apply here.
Ouimet obtained an expansion for the ratio of a Dirichlet density to the
Gaussian density with the same mean and covariance, together with a bound in
total variation \refcitenote{Theorems~1--2}{Ouimet2022}. Returning to the
$\theta$ coordinate, the matching Gaussian variance is
$\pi^2/[4(N+3)]$. Henzi and D\"umbgen also compared symmetric beta densities
with Gaussian densities by
bounding their density ratios
\refcitenote{Section~5}{HenziDuembgenExtended2023}. Their argument uses the
general method of D\"umbgen, Samworth, and Wellner
\refcite{DuembgenSamworthWellner2021}. Two of the choices considered by Henzi
and D\"umbgen are the exact standard deviation and the scale obtained from the
quadratic term of the log density at its maximum. In our notation, these
scales are
\[
\frac{\pi}{2\sqrt{N+3}}
\qquad\text{and}\qquad
\frac{\pi}{2\sqrt N},
\]
respectively.

The preceding results provide Gaussian approximations for both rescalings.
For each rescaling, we identify the first-order correction, of order
$N^{-1}$, to the standard Gaussian density in $L^1(\R)$; see
\iref{thm:ball-first-order-density} and
\iref{thm:ball-exact-variance-first-order}. Integrating the limiting functions
over symmetric intervals gives first-order corrections to the probabilities
of these intervals. Integrating their absolute values over $\R$ and dividing
by two gives the leading terms of the total variation distances; see
\Cref{thm:ball-sharp-TV} and
\iref{cor:ball-exact-variance-consequences}.

We first establish a general approximation result for a sequence of continuous
profiles $\rho_n$ on a fixed interval.
\iref{prop:tail-to-core} gives an upper bound for the probability outside a
neighborhood of the set of maximizers.
\iref{prop:fixed-scale-event-comparison} then compares the probability of that
neighborhood under the normalized profile with the probability that
$\Theta_n(X_n)$ lies in the same neighborhood, where
$X_n\sim N(\mu_n,\Sigma_n)$. When $\rho_n$ has a unique nondegenerate
maximizer $\alpha_n$, we set
\[
\tau_n=[-(\log\rho_n)''(\alpha_n)]^{-1/2}.
\]
This is the scale suggested by the second-order Taylor expansion of
$\log\rho_n$ at $\alpha_n$. We center at $\alpha_n$ and introduce the variable
$y=(\theta-\alpha_n)/\tau_n$. Under the assumptions of
\iref{thm:local-gaussian-approximation}, the density of $y$ converges in
$L^1(\R)$ to $\varphi(y)=(2\pi)^{-1/2}e^{-y^2/2}$. The proof combines a
Taylor expansion on bounded intervals with an estimate of the integral over
$|y|>M$. This is the usual argument in Laplace's method.

We next apply this result to the profile of the Euclidean ball
\eqref{eq:intro-ball-profile}. Here $\alpha_n=0$ and
$\tau_n=\pi/(2\sqrt N)$. After centering at zero and rescaling by $\tau_n$,
we obtain the density
\[
p_n(y)=\frac1{D_N}
\begin{cases}
\displaystyle\left(1-\frac{y^2}{N}\right)^{N/2},
& |y|\le\sqrt N,\\[3pt]
0, & |y|>\sqrt N,
\end{cases}
\]
where $D_N$ is the normalizing constant. We prove
\begin{equation}
N(p_n-\varphi)
\longrightarrow
Q(y):=\frac{3-y^4}{4}\varphi(y)
\quad\text{in }L^1(\R),
\label{eq:intro-first-order-density}
\end{equation}
or, equivalently,
\[
\|p_n-\varphi-N^{-1}Q\|_{L^1(\R)}=o(N^{-1}).
\]
It follows that
\begin{equation}
N d_{\rm TV}(p_n,\varphi)
\longrightarrow
C_{\rm TV}:=\frac{(a^3+3a)\varphi(a)}2
\approx0.522516163,
\qquad a=3^{1/4}.
\label{eq:intro-sharp-tv}
\end{equation}
For each fixed $x\ge0$, integrating $Q$ over $[-x,x]$ shows that
$H(x):=\frac{x^3+3x}{2}\varphi(x)$ is the first-order correction to the
Gaussian probability of this interval. Moreover, the local expansion and the
monotonicity of $p_n/\varphi$ show that the equation
$p_n(y)=\varphi(y)$ has a unique positive solution $r_n$ and that $r_n\to a$.

The second rescaling uses the exact standard deviation of the marginal
distribution and therefore gives the comparison with the Gaussian
distribution of the same variance in \eqref{eq:intro-brehm-voigt}. Set
\[
\kappa_n=\sqrt{\frac{N}{N+3}},
\qquad
\widehat p_n(z)=\kappa_np_n(\kappa_nz).
\]
Then \iref{thm:ball-exact-variance-first-order} and
\iref{cor:ball-exact-variance-consequences} give
\begin{align}
N(\widehat p_n-\varphi)
&\longrightarrow
\widehat Q(z):=-\frac{z^4-6z^2+3}{4}\varphi(z)
&&\text{in }L^1(\R),
\label{eq:intro-exact-variance-first-order}\\
N d_{\rm TV}(\widehat p_n,\varphi)
&\longrightarrow
\widehat C_{\rm TV},
\notag\\
\widehat C_{\rm TV}
&:=\frac{\sqrt6}{2}
\bigl[b_-\varphi(b_-)+b_+\varphi(b_+)\bigr]
\approx0.350075038,
\label{eq:intro-exact-variance-tv}
\end{align}
where $b_\pm=\sqrt{3\pm\sqrt6}$. For each fixed $x\ge0$, integrating
$\widehat Q$ over $[-x,x]$ shows that
$\widehat H(x)=\frac{x^3-3x}{2}\varphi(x)$ is the first-order correction to
the Gaussian probability of this interval.

The two correction functions can also be expressed in terms of Hermite
polynomials. Let
\[
H_2(t)=t^2-1,
\qquad
H_4(t)=t^4-6t^2+3
\]
be the probabilists' Hermite polynomials of degrees two and four. Then
\[
Q=-\frac14H_4\varphi-\frac32H_2\varphi,
\qquad
\widehat Q=-\frac14H_4\varphi.
\]
Thus, rescaling by the exact standard deviation removes the $H_2\varphi$ term
from the first-order correction.

The remainder of the paper is organized as follows.
\Cref{sec:fixed} proves a concentration estimate for normalized profile
functions. \Cref{sec:gaussian} studies the same concentration event when
$X_n\sim N(\mu_n,\Sigma_n)$. \Cref{sec:local} proves $L^1$ convergence after
centering and rescaling at the maximizer. \Cref{sec:ball} applies these results
to one-dimensional marginals of Euclidean balls and derives the first-order
expansions for the two rescalings.


\section{Profile functions and fixed-scale concentration}
\label{sec:fixed}

Let \(I=[a,b]\subset\R\) be fixed. For each \(n\in\N\), let
\(\rho_n:I\to[0,\infty)\) be continuous and satisfy
\(0<Z_n:=\int_I\rho_n(\theta)\dd\theta<\infty\). Define the Borel
probability measure \(\nu_n(B):=Z_n^{-1}\int_{B\cap I}\rho_n(\theta)\dd\theta\), \(B\in\mathcal B(\R)\), supported on \(I\).
Set \(m_n:=\max_I\rho_n\) and
\(A_n:=\{\theta\in I:\rho_n(\theta)=m_n\}\). Then \(m_n>0\), and
\(A_n\) is nonempty and compact.

For \(\delta\ge0\), put
\[
\widetilde U_\delta(A_n)
:=\{\theta\in\R:\dist(\theta,A_n)<\delta\},
\qquad
U_\delta^I(A_n):=I\cap\widetilde U_\delta(A_n),
\]
and define
\begin{equation}
f_n(\delta)
:=\nu_n\bigl(\widetilde U_\delta(A_n)\bigr)
=\nu_n\bigl(U_\delta^I(A_n)\bigr)
=\frac{\int_{U_\delta^I(A_n)}\rho_n(\theta)\dd\theta}
{\int_I\rho_n(\theta)\dd\theta}.
\label{eq:profile-fn}
\end{equation}
In particular, \(\widetilde U_0(A_n)=U_0^I(A_n)=\varnothing\) and
\(f_n(0)=0\).

For \(\delta>0\) and \(c\in(0,1)\), set
\[
L_{n,\delta}^{(c)}
:=\{\theta\in U_\delta^I(A_n):\rho_n(\theta)\ge cm_n\},
\qquad
T_{n,\delta}:=I\setminus U_\delta^I(A_n).
\]
If \(T_{n,\delta}=\varnothing\), then \(f_n(\delta)=1\), and all tail
quantities below are understood to be zero.

\begin{proposition}
\midlabel{proposition}{prop:tail-to-core}
For every \(n\in\N\), \(\delta>0\), and \(c\in(0,1)\), one has
\(|L_{n,\delta}^{(c)}|>0\) and
\begin{equation}
1-f_n(\delta)
\le
\frac{|T_{n,\delta}|}{c|L_{n,\delta}^{(c)}|}
\sup_{\theta\in T_{n,\delta}}\frac{\rho_n(\theta)}{m_n}
\le
\frac{|I|}{c|L_{n,\delta}^{(c)}|}
\sup_{\theta\in T_{n,\delta}}\frac{\rho_n(\theta)}{m_n}.
\label{eq:tail-to-core-estimate}
\end{equation}
\end{proposition}

\begin{proof}
Choose \(\alpha_n\in A_n\). Since \(\rho_n(\alpha_n)=m_n>cm_n\),
continuity gives \(r_n>0\) such that \(\theta\in I\) and
\(|\theta-\alpha_n|<r_n\) imply \(\rho_n(\theta)>cm_n\). Replacing
\(r_n\) by \(\min\{r_n,\delta/2\}\), we obtain
\((\alpha_n-r_n,\alpha_n+r_n)\cap I\subset L_{n,\delta}^{(c)}\),
because \(\dist(\theta,A_n)\le|\theta-\alpha_n|<\delta\) there. Hence \(|L_{n,\delta}^{(c)}|>0\).

If \(T_{n,\delta}=\varnothing\), the result follows from the convention
above. Otherwise,
\[
\int_{U_\delta^I(A_n)}\rho_n(\theta)\dd\theta
\ge cm_n|L_{n,\delta}^{(c)}|,
\qquad
\int_{T_{n,\delta}}\rho_n(\theta)\dd\theta
\le |T_{n,\delta}|\sup_{\theta\in T_{n,\delta}}\rho_n(\theta).
\]
Therefore the tail-to-core mass ratio satisfies
\begin{equation}
r_{n,\delta}
:=\frac{\int_{T_{n,\delta}}\rho_n(\theta)\dd\theta}
{\int_{U_\delta^I(A_n)}\rho_n(\theta)\dd\theta}
\le
\frac{|T_{n,\delta}|}{c|L_{n,\delta}^{(c)}|}
\sup_{\theta\in T_{n,\delta}}\frac{\rho_n(\theta)}{m_n}.
\label{eq:tail-mass-ratio}
\end{equation}
Since \(I=U_\delta^I(A_n)\,\dot\cup\,T_{n,\delta}\),
\[
1-f_n(\delta)
=\frac{\int_{T_{n,\delta}}\rho_n(\theta)\dd\theta}
{\int_{U_\delta^I(A_n)}\rho_n(\theta)\dd\theta
+\int_{T_{n,\delta}}\rho_n(\theta)\dd\theta}
=\frac{r_{n,\delta}}{1+r_{n,\delta}}
\le r_{n,\delta}.
\]
The first inequality follows from \eqref{eq:tail-mass-ratio}, and the
second follows from \(|T_{n,\delta}|\le|I|\).
\end{proof}

For \(\delta>0\) and \(c\in(0,1)\), define
\[
\eta_{n,\delta}^{(c)}
:=\frac{|T_{n,\delta}|}{c|L_{n,\delta}^{(c)}|}
\sup_{\theta\in T_{n,\delta}}\frac{\rho_n(\theta)}{m_n}.
\]
Then
\begin{equation}
0\le1-f_n(\delta)\le\eta_{n,\delta}^{(c)}.
\label{eq:fn-eta-bound}
\end{equation}

\begin{corollary}
\midlabel{corollary}{cor:fixed-scale-concentration}
Fix \(c\in(0,1)\). If \(\eta_{n,\delta}^{(c)}\to0\) for every fixed
\(\delta>0\), then \(f_n(\delta)\to1\) for every fixed \(\delta>0\).
\end{corollary}

\begin{proof}
For fixed \(\delta>0\), \(0\le1-f_n(\delta)\le\eta_{n,\delta}^{(c)}\to0\) and the result follows.
\end{proof}

\begin{corollary}
\label{cor:fixed-scale-simple}
Fix \(c\in(0,1)\). If, for every fixed \(\delta>0\),
\begin{equation}
\sup_{\theta\in T_{n,\delta}}\frac{\rho_n(\theta)}{m_n}
\longrightarrow0,
\qquad
\liminf_{n\to\infty}|L_{n,\delta}^{(c)}|>0,
\label{eq:fixed-scale-sufficient}
\end{equation}
then \(f_n(\delta)\to1\) for every fixed \(\delta>0\).
\end{corollary}

\begin{proof}
Fix \(\delta>0\) and choose \(0<\ell_\delta<\liminf_{n\to\infty}|L_{n,\delta}^{(c)}|\). Then \(|L_{n,\delta}^{(c)}|\ge\ell_\delta\) for all sufficiently large \(n\), and hence
\[
0\le\eta_{n,\delta}^{(c)}
\le\frac{|I|}{c\ell_\delta}
\sup_{\theta\in T_{n,\delta}}\frac{\rho_n(\theta)}{m_n}
\longrightarrow0.
\]
The result follows from
\hyperlink{mid:cor:fixed-scale-concentration}{\mbox{\textcolor{internalblue}{\textbf{Corollary~\ref*{cor:fixed-scale-concentration}}}}}.
\end{proof}

\begin{remark}
Decay of \(\sup_{\theta\in T_{n,\delta}}\rho_n(\theta)/m_n\) alone is
insufficient: the factor \(|L_{n,\delta}^{(c)}|^{-1}\) in
\(\eta_{n,\delta}^{(c)}\) may diverge. Thus the tail-height decay must
be accompanied by quantitative control of the core width.
\end{remark}

Because \(\nu_n\) is supported on \(I\), \Cref{sec:gaussian} can use the same ambient neighborhood \(\widetilde U_\delta(A_n)\) for the Gaussian law; density shape is treated at the shrinking scale in \Cref{sec:local}.


\section{Fixed-scale comparison of concentration events}
\label{sec:gaussian}

Let \(X_n\sim N(\mu_n,\Sigma_n)\), where \(\mu_n\in\R^n\) and
\(\Sigma_n\) is symmetric positive semidefinite. Let
\(\Theta_n:\R^n\to\R\) be \(C^1\) with
\(K_n:=\sup_{z\in\R^n}\|\nabla\Theta_n(z)\|<\infty\).
For \(\delta\ge0\), define
\begin{equation}
g_n(\delta):=\P\bigl(\Theta_n(X_n)\in\widetilde U_\delta(A_n)\bigr)
=\P\bigl(\dist(\Theta_n(X_n),A_n)<\delta\bigr).
\label{eq:gaussian-gn}
\end{equation}
In particular, \(g_n(0)=0\). Set
\(d_n:=\dist(\Theta_n(\mu_n),A_n)\).

\subsection{A Gaussian observation bound}

\begin{lemma}
\label{lem:gaussian-distance}
For every \(x\in\R^n\),
\begin{equation}
\dist(\Theta_n(x),A_n)\le K_n\|x-\mu_n\|+d_n.
\label{eq:gaussian-pointwise-distance}
\end{equation}
Consequently,
\begin{equation}
\left[\E\dist(\Theta_n(X_n),A_n)^2\right]^{1/2}
\le K_n\sqrt{\tr(\Sigma_n)}+d_n,
\label{eq:gaussian-L2-distance}
\end{equation}
and
\begin{equation}
\E\dist(\Theta_n(X_n),A_n)^2
\le\left(K_n\sqrt{\tr(\Sigma_n)}+d_n\right)^2.
\label{eq:gaussian-distance-square-bound}
\end{equation}
\end{lemma}

\begin{proof}
The gradient bound and the line-segment Fundamental Theorem of Calculus give $|\Theta_n(x)-\Theta_n(y)|\le K_n\|x-y\|$. Choose $\alpha_n^*\in A_n$ with $d_n=|\Theta_n(\mu_n)-\alpha_n^*|$. Then
\[
 \dist(\Theta_n(x),A_n)\le |\Theta_n(x)-\alpha_n^*|
 \le K_n\|x-\mu_n\|+d_n.
\]
Since $\E\|X_n-\mu_n\|^2=\tr(\Sigma_n)$, Minkowski's inequality in $L^2$ gives \eqref{eq:gaussian-L2-distance}; squaring proves \eqref{eq:gaussian-distance-square-bound}.
\end{proof}

\begin{proposition}
\midlabel{proposition}{prop:gaussian-observation-estimate}
For every \(\delta>0\),
\begin{equation}
1-g_n(\delta)
\le\frac{\left(K_n\sqrt{\tr(\Sigma_n)}+d_n\right)^2}{\delta^2}.
\label{eq:gaussian-observation-estimate}
\end{equation}
\end{proposition}

\begin{proof}
By Markov's inequality and \eqref{eq:gaussian-distance-square-bound},
\begin{align*}
1-g_n(\delta)
&=\P\bigl(\dist(\Theta_n(X_n),A_n)^2\ge\delta^2\bigr)\\
&\le\frac{\E\dist(\Theta_n(X_n),A_n)^2}{\delta^2}
\le\frac{\left(K_n\sqrt{\tr(\Sigma_n)}+d_n\right)^2}{\delta^2}.\qedhere
\end{align*}
\end{proof}

\begin{corollary}
\label{cor:gaussian-fixed-concentration}
If \(K_n^2\tr(\Sigma_n)\to0\) and \(d_n\to0\), then \(g_n(\delta)\to1\) for every fixed \(\delta>0\).
\end{corollary}

\begin{proof}
Since $K_n\sqrt{\tr(\Sigma_n)}+d_n=\sqrt{K_n^2\tr(\Sigma_n)}+d_n\to0$, the result follows from \iref{prop:gaussian-observation-estimate}.
\end{proof}

\subsection{Event-level comparison with the profile mass}

\begin{proposition}
\midlabel{proposition}{prop:fixed-scale-event-comparison}
Fix \(c\in(0,1)\). For every \(\delta>0\),
\begin{equation}
|f_n(\delta)-g_n(\delta)|
\le\eta_{n,\delta}^{(c)}
+\frac{\left(K_n\sqrt{\tr(\Sigma_n)}+d_n\right)^2}{\delta^2}.
\label{eq:gaussian-comparison-bound}
\end{equation}
Consequently, if \(\eta_{n,\delta}^{(c)}\to0\) for every fixed \(\delta>0\), \(K_n^2\tr(\Sigma_n)\to0\), and \(d_n\to0\), then \(|f_n(\delta)-g_n(\delta)|\to0\) for every fixed \(\delta>0\).
\end{proposition}

\begin{proof}
Since \(0\le f_n(\delta),g_n(\delta)\le1\),
\[
|f_n(\delta)-g_n(\delta)|
\le
(1-f_n(\delta))+(1-g_n(\delta)).
\]
Using \eqref{eq:fn-eta-bound} and
\iref{prop:gaussian-observation-estimate}, we obtain
\[
|f_n(\delta)-g_n(\delta)|
\le
\eta_{n,\delta}^{(c)}
+
\frac{\left(K_n\sqrt{\tr(\Sigma_n)}+d_n\right)^2}{\delta^2},
\]
which proves \eqref{eq:gaussian-comparison-bound}.

Now suppose that
\(\eta_{n,\delta}^{(c)}\to0\) for every fixed \(\delta>0\),
\(K_n^2\tr(\Sigma_n)\to0\), and \(d_n\to0\). Since
\[
K_n\sqrt{\tr(\Sigma_n)}
=
\sqrt{K_n^2\tr(\Sigma_n)}
\to0,
\]
the second term on the right-hand side of
\eqref{eq:gaussian-comparison-bound} tends to zero for every fixed
\(\delta>0\). Hence
$|f_n(\delta)-g_n(\delta)|\to0$, as claimed.
\end{proof}

\begin{remark}
Since $\nu_n$ is supported on $I$, the profile probability
$f_n(\delta)$ is unchanged if $\widetilde U_\delta(A_n)$ is replaced by
$U_\delta^I(A_n)$. For the Gaussian observation this need not be true.
If
\[
g_n^I(\delta)
:=
\mathbb P\!\left(
\Theta_n(X_n)\in U_\delta^I(A_n)
\right),
\]
then
\[
0\le g_n(\delta)-g_n^I(\delta)
=
\mathbb P\!\left(
\Theta_n(X_n)\in
\widetilde U_\delta(A_n)\setminus I
\right)
=:\ell_{n,\delta}.
\]
Hence
\[
|f_n(\delta)-g_n^I(\delta)|
\le
|f_n(\delta)-g_n(\delta)|+\ell_{n,\delta}.
\]
Therefore, under the hypotheses of
\iref{prop:fixed-scale-event-comparison}, if
$\ell_{n,\delta}\to0$ for every fixed $\delta>0$, then
$|f_n(\delta)-g_n^I(\delta)|\to0$. The leakage term is relevant when Gaussian mass can cross the boundary
of $I$, in particular when $A_n$ meets or approaches $\partial I$.
\end{remark}

\subsection{Center and covariance choices at the fixed scale}

\begin{lemma}
\label{lem:center-selection}
Choose \(\alpha_n\in A_n\). If \(z_n^-,z_n^+\in\R^n\) satisfy
\begin{equation}
\Theta_n(z_n^-)\le\alpha_n\le\Theta_n(z_n^+),
\label{eq:center-bracketing}
\end{equation}
then there exists \(\mu_n\in\R^n\) such that \(\Theta_n(\mu_n)=\alpha_n\), and hence \(d_n=0\).
\end{lemma}

\begin{proof}
Apply the Intermediate Value Theorem to $\Theta_n(z_n^-+t(z_n^+-z_n^-))$, $0\le t\le1$, and take $\mu_n$ at the resulting parameter value.
\end{proof}

\begin{proposition}
\label{prop:selected-center-convergence}
Suppose \(\alpha_n\in A_n\cap\Theta_n(\R^n)\) and choose \(\mu_n\) so that \(\Theta_n(\mu_n)=\alpha_n\). Then \((\alpha_n)\subset I\) has a convergent subsequence. If, for some \(\alpha\in I\),
\begin{equation}
\sup_{\beta\in A_n}|\beta-\alpha|\longrightarrow0,
\label{eq:maximizer-set-collapse}
\end{equation}
then \(\alpha_n=\Theta_n(\mu_n)\to\alpha\).
\end{proposition}

\begin{proof}
Compactness of $I$ gives the subsequence, while \eqref{eq:maximizer-set-collapse} gives $|\alpha_n-\alpha|\le\sup_{\beta\in A_n}|\beta-\alpha|\to0$.
\end{proof}

\begin{remark}
Only the scalar sequence \(\Theta_n(\mu_n)=\alpha_n\) is asserted to converge; no convergence of \(\mu_n\in\R^n\) is meant when the ambient dimension varies.
\end{remark}

The following isotropic covariance is used only at the fixed scale; the superscript \((F)\) distinguishes it from the curvature-calibrated covariance below.

\begin{proposition}
\label{prop:isotropic-covariance}
Assume \(K_n>0\) and \(\lambda_n>0\), \(\lambda_n\to\infty\). Define
\begin{equation}
\Sigma_n^{(F)}:=\frac{1}{nK_n^2\lambda_n}I_n.
\label{eq:isotropic-covariance}
\end{equation}
Then \(\Sigma_n^{(F)}\) is positive definite and
\begin{equation}
\tr(\Sigma_n^{(F)})=\frac{1}{K_n^2\lambda_n},
\qquad
K_n^2\tr(\Sigma_n^{(F)})=\frac{1}{\lambda_n}\longrightarrow0,
\qquad
\|\Sigma_n^{(F)}\|_{\rm op}=\frac{1}{nK_n^2\lambda_n}.
\label{eq:isotropic-fluctuation}
\end{equation}
\end{proposition}

\begin{proof}
All eigenvalues equal $(nK_n^2\lambda_n)^{-1}>0$; their sum and maximum give \eqref{eq:isotropic-fluctuation}.
\end{proof}

When the maximizer is unique, \(\lambda_n\) is naturally supplied by the curvature of \(\log\rho_n\).

\begin{lemma}
\label{lem:curvature-rho}
Assume \(A_n=\{\alpha_n\}\), \(\alpha_n\in\operatorname{int}(I)\), \(\rho_n(\alpha_n)=m_n>0\), and \(\rho_n\) is twice differentiable near \(\alpha_n\). Then \(\rho_n'(\alpha_n)=0\) and
\begin{equation}
-\bigl(\log\rho_n\bigr)''(\alpha_n)=-\frac{\rho_n''(\alpha_n)}{m_n}.
\label{eq:curvature-rho-second}
\end{equation}
Suppose that $c_0>0$ and $b_n\to\infty$. If
\(-\rho_n''(\alpha_n)\ge c_0b_nm_n\), then
\(\lambda_n:=-\bigl(\log\rho_n\bigr)''(\alpha_n)\to\infty\).
\end{lemma}

\begin{proof}
Fermat's theorem gives \(\rho_n'(\alpha_n)=0\). Since \((\log\rho_n)''=\rho_n''/\rho_n-(\rho_n'/\rho_n)^2\), evaluation at \(\alpha_n\) gives \eqref{eq:curvature-rho-second}; hence \(\lambda_n=-\rho_n''(\alpha_n)/m_n\ge c_0b_n\to\infty\).
\end{proof}

\begin{corollary}
\label{cor:constructive-comparison}
Fix \(c\in(0,1)\). Suppose \(\eta_{n,\delta}^{(c)}\to0\) for every fixed \(\delta>0\), \(K_n>0\), \(\lambda_n\to\infty\), and \(\Theta_n(\mu_n)\in A_n\) for some \(\mu_n\in\R^n\). If \(\Sigma_n=\Sigma_n^{(F)}\), then \(|f_n(\delta)-g_n(\delta)|\to0\) for every fixed \(\delta>0\).
\end{corollary}

\begin{proof}
Since \(\Theta_n(\mu_n)\in A_n\), we have
$d_n=\dist(\Theta_n(\mu_n),A_n)=0$. Moreover, by the definition of \(\Sigma_n^{(F)}\),
$K_n^2\tr(\Sigma_n^{(F)})
=\lambda_n^{-1}\to0$. Thus all the hypotheses of \iref{prop:fixed-scale-event-comparison} are satisfied, and hence
$|f_n(\delta)-g_n(\delta)|\to0$ for every fixed \(\delta>0\).
\end{proof}

\section{Local Gaussian approximation}
\label{sec:local}

Assume throughout that \(I=[a,b]\subset\R\), \(A_n=\{\alpha_n\}\), \(\alpha_n\in(a,b)\), \(\rho_n>0\) on \((a,b)\), \(V_n:=\log\rho_n\in C^2((a,b))\), and \(V_n''(\alpha_n)<0\).
Since \(\alpha_n\) is an interior maximizer, \(V_n'(\alpha_n)=0\). Set
\begin{equation}
\lambda_n:=-V_n''(\alpha_n)>0,
\qquad
\tau_n:=\lambda_n^{-1/2}.
\label{eq:local-tau}
\end{equation}
Thus \(\lambda_n\to\infty\) is equivalent to \(\tau_n\to0\).

\subsection{Local quadratic approximation}

\begin{lemma}
\midlabel{lemma}{lem:integral-taylor-remainder}
Let \(V\in C^2((a,b))\), \(\alpha,\alpha+h\in(a,b)\). Then
\begin{equation}
\begin{aligned}
V(\alpha+h)
&=V(\alpha)+V'(\alpha)h+\frac12V''(\alpha)h^2\\
&\quad+h^2\int_0^1(1-t)\bigl[V''(\alpha+th)-V''(\alpha)\bigr]\dd t.
\end{aligned}
\label{eq:integral-taylor-remainder}
\end{equation}
Consequently,
\begin{equation}
V(\alpha+h)=V(\alpha)+V'(\alpha)h+\frac12V''(\alpha)h^2+o(h^2)
\qquad(h\to0).
\label{eq:taylor-o-h2}
\end{equation}
\end{lemma}

\begin{proof}
Twice applying the Fundamental Theorem of Calculus gives $V(\alpha+h)-V(\alpha)=V'(\alpha)h+\int_0^h(h-r)V''(\alpha+r)\dd r$. Add and subtract $V''(\alpha)$ and set $r=th$ to obtain \eqref{eq:integral-taylor-remainder}. Continuity of $V''$ at $\alpha$ makes the integral factor in its remainder $o(1)$, proving \eqref{eq:taylor-o-h2}.
\end{proof}

For \(M>0\), define
\begin{equation}
\omega_n(M):=\frac1{\lambda_n}
\sup_{\substack{\theta\in(a,b)\\|\theta-\alpha_n|\le M\tau_n}}
|V_n''(\theta)-V_n''(\alpha_n)|.
\label{eq:local-curvature-variation}
\end{equation}
We assume the relative-curvature condition
\begin{equation}
\omega_n(M)\longrightarrow0
\qquad\text{for every fixed }M>0,
\label{eq:local-curvature-condition}
\end{equation}
and the asymptotic-interiority condition
\begin{equation}
\frac{\dist(\alpha_n,\partial I)}{\tau_n}\longrightarrow\infty.
\label{eq:local-interior-condition}
\end{equation}

\begin{lemma}
\label{lem:uniform-interior}
If \(\dist(\alpha_n,\partial I)\ge d_0>0\) for every \(n\) and \(\lambda_n\to\infty\), then \eqref{eq:local-interior-condition} holds.
\end{lemma}

\begin{proof}
Indeed, \(\dist(\alpha_n,\partial I)/\tau_n=\sqrt{\lambda_n}\,\dist(\alpha_n,\partial I)\ge d_0\sqrt{\lambda_n}\to\infty\).
\end{proof}

\begin{theorem}
\midlabel{theorem}{thm:exact-local-neighborhood}
Under \eqref{eq:local-interior-condition}, for every fixed \(x>0\),
\begin{equation}
U_{\tau_nx}^I(A_n)=(\alpha_n-\tau_nx,\alpha_n+\tau_nx)
\label{eq:exact-local-neighborhood}
\end{equation}
for all sufficiently large \(n\).
\end{theorem}

\begin{proof}
For large \(n\), \(\tau_nx<\dist(\alpha_n,\partial I)\); hence \((\alpha_n-\tau_nx,\alpha_n+\tau_nx)\subset(a,b)\). Since \(A_n=\{\alpha_n\}\), the defining inequality \(\dist(\theta,A_n)<\tau_nx\) is exactly \(|\theta-\alpha_n|<\tau_nx\).
\end{proof}

Set
\[
J_n:=\{y\in\R:\alpha_n+\tau_ny\in I\}
=\left[\frac{a-\alpha_n}{\tau_n},\frac{b-\alpha_n}{\tau_n}\right]
\]
and
\begin{equation}
h_n(y):=
\begin{cases}
\displaystyle\frac{\rho_n(\alpha_n+\tau_ny)}{m_n},
& y\in J_n,\\[4pt]
0, & y\notin J_n.
\end{cases}
\label{eq:local-hn}
\end{equation}
By \eqref{eq:local-interior-condition}, every fixed \([-M,M]\) is contained in \(J_n\) for all sufficiently large \(n\).

\begin{lemma}
\midlabel{lemma}{lem:local-quadratic-approximation}
Assume \eqref{eq:local-interior-condition} and
\eqref{eq:local-curvature-condition}. For every fixed \(M>0\), the following
estimate holds for all sufficiently large \(n\):
\begin{equation}
\sup_{|y|\le M}
\left|V_n(\alpha_n+\tau_ny)-V_n(\alpha_n)+\frac{y^2}{2}\right|
\le\frac{M^2}{2}\omega_n(M)\longrightarrow0.
\label{eq:local-log-bound}
\end{equation}
\end{lemma}

\begin{proof}
For large \(n\), \([\alpha_n-M\tau_n,\alpha_n+M\tau_n]\subset(a,b)\). Apply \iref{lem:integral-taylor-remainder} with \(h=\tau_ny\), use \(V_n'(\alpha_n)=0\), \(V_n''(\alpha_n)=-\lambda_n\), and \(\lambda_n\tau_n^2=1\):
\begin{align*}
&V_n(\alpha_n+\tau_ny)-V_n(\alpha_n)+\frac{y^2}{2}\\
&\quad=\tau_n^2y^2\int_0^1(1-t)
[V_n''(\alpha_n+t\tau_ny)-V_n''(\alpha_n)]\dd t.
\end{align*}
For \(|y|\le M\), \(|t\tau_ny|\le M\tau_n\), so
\begin{align*}
\left|V_n(\alpha_n+\tau_ny)-V_n(\alpha_n)+\frac{y^2}{2}\right|
&\le\tau_n^2y^2\lambda_n\omega_n(M)\int_0^1(1-t)\dd t\\
&=\frac{y^2}{2}\omega_n(M)
\le\frac{M^2}{2}\omega_n(M).\qedhere
\end{align*}
\end{proof}

\begin{corollary}
\midlabel{corollary}{cor:local-profile-limit}
Under the preceding assumptions, for every fixed \(M>0\),
\begin{equation}
\sup_{|y|\le M}|h_n(y)-e^{-y^2/2}|\longrightarrow0.
\label{eq:local-profile-uniform}
\end{equation}
\end{corollary}

\begin{proof}
For large \(n\), \([-M,M]\subset J_n\), and \(h_n(y)=\exp\{V_n(\alpha_n+\tau_ny)-V_n(\alpha_n)\}\). The exponents converge uniformly to \(-y^2/2\) by \iref{lem:local-quadratic-approximation}; exponentiation preserves this uniform convergence on \([-M,M]\).
\end{proof}

A third-derivative bound gives a convenient sufficient condition for the same conclusion; the assumption \(\lambda_n\to\infty\) alone does not.

\begin{proposition}
\label{prop:third-derivative-condition}
Assume \eqref{eq:local-interior-condition} and suppose that,
for every fixed \(M>0\), \(V_n\in C^3\) on an open neighborhood of
\([\alpha_n-M\tau_n,\alpha_n+M\tau_n]\) for all sufficiently large \(n\), and
\begin{equation}
\tau_n^3
\sup_{|\theta-\alpha_n|\le M\tau_n}
|V_n'''(\theta)|
\longrightarrow0.
\label{eq:third-derivative-smallness}
\end{equation}
Then \eqref{eq:local-curvature-condition} holds, and hence
\eqref{eq:local-log-bound} and \eqref{eq:local-profile-uniform} hold.
\end{proposition}

\begin{proof}
For all sufficiently large \(n\), the Mean Value Theorem gives
\[
\omega_n(M)
\le \frac{M\tau_n}{\lambda_n}
\sup_{|\theta-\alpha_n|\le M\tau_n}|V_n'''(\theta)|
=M\tau_n^3
\sup_{|\theta-\alpha_n|\le M\tau_n}|V_n'''(\theta)|
\longrightarrow0.
\]
Thus \eqref{eq:local-curvature-condition} holds. The conclusions now follow
from \iref{lem:local-quadratic-approximation} and
\iref{cor:local-profile-limit}.
\end{proof}

Since $\tau_n=\lambda_n^{-1/2}$, condition \eqref{eq:third-derivative-smallness} is $\sup_{|\theta-\alpha_n|\le M/\sqrt{\lambda_n}}|V_n'''(\theta)|/\lambda_n^{3/2}\to0$; in particular, an $O(\lambda_n)$ bound on this window is sufficient.

\subsection{Local numerator and tail control}

\begin{proposition}
\label{prop:local-numerator}
Under \eqref{eq:local-interior-condition} and \eqref{eq:local-curvature-condition}, for every fixed \(x>0\),
\begin{equation}
\frac1{m_n\tau_n}\int_{U_{\tau_nx}^I(A_n)}\rho_n(\theta)\dd\theta
\longrightarrow\int_{-x}^{x}e^{-y^2/2}\dd y.
\label{eq:local-numerator-limit}
\end{equation}
\end{proposition}

\begin{proof}
For large \(n\), \iref{thm:exact-local-neighborhood} and \(\theta=\alpha_n+\tau_ny\) give
\begin{equation}
\frac1{m_n\tau_n}\int_{U_{\tau_nx}^I(A_n)}\rho_n(\theta)\dd\theta
=\int_{-x}^{x}\frac{\rho_n(\alpha_n+\tau_ny)}{m_n}\dd y.
\label{eq:local-numerator-change-variable}
\end{equation}
The integrand converges uniformly on \([-x,x]\) to \(e^{-y^2/2}\) by \iref{cor:local-profile-limit}, proving the claim.
\end{proof}

Local convergence must be supplemented by rescaled tail tightness:
\begin{equation}
\lim_{M\to\infty}\limsup_{n\to\infty}
\frac1{m_n\tau_n}
\int_{\substack{\theta\in I\\|\theta-\alpha_n|>M\tau_n}}\rho_n(\theta)\dd\theta=0,
\label{eq:local-tail-tightness}
\end{equation}
equivalently,
\begin{equation}
\lim_{M\to\infty}\limsup_{n\to\infty}\int_{|y|>M}h_n(y)\dd y=0.
\label{eq:local-tail-tightness-rescaled}
\end{equation}

\begin{lemma}
\label{lem:global-curvature-tail}
Assume \(V_n'(\alpha_n)=0\) and, for some \(c_0>0\),
\begin{equation}
V_n''(\theta)\le-\frac{c_0}{\tau_n^2}
\qquad(\theta\in(a,b)).
\label{eq:global-curvature-upper}
\end{equation}
Then
\begin{equation}
\frac{\rho_n(\theta)}{m_n}
\le\exp\!\left[-\frac{c_0}{2\tau_n^2}(\theta-\alpha_n)^2\right]
\qquad(\theta\in I),
\label{eq:global-profile-gaussian-upper}
\end{equation}
and
\begin{equation}
\frac1{m_n\tau_n}
\int_{\substack{\theta\in I\\|\theta-\alpha_n|\ge M\tau_n}}\rho_n(\theta)\dd\theta
\le\int_{|y|\ge M}e^{-c_0y^2/2}\dd y.
\label{eq:global-tail-gaussian-bound}
\end{equation}
Hence \eqref{eq:local-tail-tightness} holds.
\end{lemma}

\begin{proof}
For \(\theta\in(a,b)\), Taylor's theorem gives \(\xi\) between \(\alpha_n\) and \(\theta\) such that
\[
V_n(\theta)-V_n(\alpha_n)
=\frac12V_n''(\xi)(\theta-\alpha_n)^2
\le-\frac{c_0}{2\tau_n^2}(\theta-\alpha_n)^2.
\]
Exponentiation yields \eqref{eq:global-profile-gaussian-upper}; continuity extends it to endpoints when needed. Then, with \(y=(\theta-\alpha_n)/\tau_n\),
\begin{align*}
\frac1{m_n\tau_n}\int_{|\theta-\alpha_n|\ge M\tau_n}\rho_n(\theta)\dd\theta
&\le\frac1{\tau_n}\int_{|\theta-\alpha_n|\ge M\tau_n}
 e^{-c_0(\theta-\alpha_n)^2/(2\tau_n^2)}\dd\theta\\
&\le\int_{|y|\ge M}e^{-c_0y^2/2}\dd y.
\end{align*}
The Gaussian tail on the right tends to zero as $M\to\infty$, proving \eqref{eq:local-tail-tightness}.
\end{proof}

\subsection{Global \texorpdfstring{\(L^1\)}{L1}-convergence}

\begin{theorem}
\label{thm:local-unnormalized-l1}
Assume \eqref{eq:local-interior-condition}, \eqref{eq:local-curvature-condition}, and \eqref{eq:local-tail-tightness}. Then
\begin{equation}
\int_\R|h_n(y)-e^{-y^2/2}|\dd y\longrightarrow0.
\label{eq:local-hn-l1}
\end{equation}
\end{theorem}

\begin{proof}
Given $\varepsilon>0$, choose $M$ so that the $h_n$-tail limsup in \eqref{eq:local-tail-tightness-rescaled} plus the Gaussian tail is below $\varepsilon$. On $[-M,M]$, \iref{cor:local-profile-limit} gives
\[
 \int_{-M}^{M}|h_n-e^{-y^2/2}|\dd y
 \le2M\sup_{|y|\le M}|h_n(y)-e^{-y^2/2}|\longrightarrow0.
\]
Splitting the integral into this central part and the two tails yields \eqref{eq:local-hn-l1}.
\end{proof}

The exact change of variables \(\theta=\alpha_n+\tau_ny\) gives
\begin{equation}
\frac{Z_n}{m_n\tau_n}=\int_\R h_n(y)\dd y.
\label{eq:local-normalization-identity}
\end{equation}

\begin{corollary}
\midlabel{corollary}{cor:local-normalization}
Under the preceding assumptions,
\begin{equation}
\frac{Z_n}{m_n\tau_n}\longrightarrow\sqrt{2\pi},
\qquad
Z_n=m_n\tau_n(\sqrt{2\pi}+o(1)).
\label{eq:local-normalization-asymptotic}
\end{equation}
\end{corollary}

\begin{proof}
By \eqref{eq:local-normalization-identity},
\[
\left|\frac{Z_n}{m_n\tau_n}-\sqrt{2\pi}\right|
\le\int_\R|h_n(y)-e^{-y^2/2}|\dd y\longrightarrow0.\qedhere
\]
\end{proof}

Define
\[
p_n(y):=\frac{h_n(y)}{c_n},
\qquad
c_n:=\int_\R h_n(s)\dd s,
\]
so \(\int_\R p_n=1\). Write \(\varphi(y):=(2\pi)^{-1/2}e^{-y^2/2}\) and \(\Phi(x):=\int_{-\infty}^{x}\varphi(y)\dd y\).

\begin{theorem}
\midlabel{theorem}{thm:local-gaussian-approximation}
Under \eqref{eq:local-interior-condition}, \eqref{eq:local-curvature-condition}, and \eqref{eq:local-tail-tightness},
\begin{equation}
\int_\R|p_n(y)-\varphi(y)|\dd y\longrightarrow0,
\qquad
d_{\rm TV}(p_n,\varphi)
:=\frac12\int_\R|p_n(y)-\varphi(y)|\dd y\longrightarrow0.
\label{eq:local-pn-l1}
\end{equation}
\end{theorem}

\begin{proof}
By \iref{cor:local-normalization}, \(c_n\to\sqrt{2\pi}>0\). Hence
\begin{align*}
\int_\R|p_n(y)-\varphi(y)|\dd y
&\le\frac1{c_n}\int_\R|h_n(y)-e^{-y^2/2}|\dd y
+\left|\frac1{c_n}-\frac1{\sqrt{2\pi}}\right|\int_\R e^{-y^2/2}\dd y\\
&=\frac1{c_n}\int_\R|h_n(y)-e^{-y^2/2}|\dd y
+\sqrt{2\pi}\left|\frac1{c_n}-\frac1{\sqrt{2\pi}}\right|\longrightarrow0.
\end{align*}
The total-variation limit is immediate from its definition.
\end{proof}

\subsection{Local masses and affine Gaussian calibration}

\begin{corollary}
\midlabel{corollary}{cor:local-symmetric-mass}
Under the assumptions of \iref{thm:local-gaussian-approximation},
\begin{equation}
\sup_{x\ge0}\left|f_n(\tau_nx)-[2\Phi(x)-1]\right|\longrightarrow0.
\label{eq:local-mass-uniform}
\end{equation}
\end{corollary}

\begin{proof}
The change of variables \(\theta=\alpha_n+\tau_ny\), with the indicator \(\mathbf1_{J_n}\) already included in \(p_n\), gives for every \(x\ge0\)
\[
f_n(\tau_nx)=\int_{-x}^{x}p_n(y)\dd y,
\qquad
2\Phi(x)-1=\int_{-x}^{x}\varphi(y)\dd y.
\]
Therefore
\[
\left|f_n(\tau_nx)-[2\Phi(x)-1]\right|
\le\int_{-x}^{x}|p_n-\varphi|\dd y
\le\int_\R|p_n-\varphi|\dd y,
\]
and the last term tends to zero independently of \(x\).
\end{proof}

Since \(I\) is fixed and bounded, the asymptotic-interiority condition
\eqref{eq:local-interior-condition} forces \(\tau_n\to0\). Hence
\eqref{eq:local-mass-uniform}, applied with \(x=\delta/\tau_n\) for fixed
\(\delta>0\), gives the first implication below; the second follows directly
from the same estimate:
\begin{equation}
\delta>0\text{ fixed}\Longrightarrow f_n(\delta)\to1,
\qquad
x\ge0\text{ fixed and }\delta_n=\tau_nx
\Longrightarrow f_n(\delta_n)\to2\Phi(x)-1.
\label{eq:two-concentration-scales}
\end{equation}

For an affine observation, the variance condition below can always be realized by a positive definite covariance. Let \(a_n\neq0\), put \(e_n:=a_n/\|a_n\|\), and recall that \(\tau_n^2=\lambda_n^{-1}\).

\begin{proposition}
\midlabel{proposition}{prop:directional-covariance}
For any \(\varepsilon_n>0\), define
\begin{equation}
\Sigma_n^{(L)}
:=\frac{1}{\lambda_n\|a_n\|^2}e_ne_n^{\mathsf T}
+\varepsilon_n(I_n-e_ne_n^{\mathsf T}).
\label{eq:directional-covariance}
\end{equation}
Then \(\Sigma_n^{(L)}\) is positive definite and
\begin{equation}
a_n^{\mathsf T}\Sigma_n^{(L)}a_n=\tau_n^2=\frac1{\lambda_n},
\qquad
\tr(\Sigma_n^{(L)})=\frac{1}{\lambda_n\|a_n\|^2}+(n-1)\varepsilon_n.
\label{eq:exact-directional-variance}
\end{equation}
\end{proposition}

\begin{proof}
Write \(w=\inner{w}{e_n}e_n+w_\perp\) with \(w_\perp\perp e_n\). Then
\[
w^{\mathsf T}\Sigma_n^{(L)}w
=\frac{\inner{w}{e_n}^2}{\lambda_n\|a_n\|^2}+\varepsilon_n\|w_\perp\|^2>0
\qquad(w\neq0),
\]
so \(\Sigma_n^{(L)}\) is positive definite. Since $a_n=\|a_n\|e_n$, $e_ne_n^{\mathsf T}$ and $I_n-e_ne_n^{\mathsf T}$ are the corresponding orthogonal projections; their traces are $1$ and $n-1$, which gives \eqref{eq:exact-directional-variance}.
\end{proof}

\begin{proposition}
\label{prop:affine-local-calibration}
Suppose
\[
\Theta_n(z)=\alpha_n+\inner{a_n}{z-\mu_n},
\qquad a_n\neq0,
\]
\(X_n\sim N(\mu_n,\Sigma_n)\), and
\begin{equation}
a_n^{\mathsf T}\Sigma_na_n=\tau_n^2.
\label{eq:affine-variance-calibration}
\end{equation}
Then
\begin{equation}
\Theta_n(X_n)\sim N(\alpha_n,\tau_n^2).
\label{eq:affine-image-normal}
\end{equation}
Moreover,
\begin{equation}
g_n(\tau_nx)=2\Phi(x)-1,
\qquad x\ge0.
\label{eq:affine-gn-local-exact}
\end{equation}
If, in addition, the assumptions of \iref{thm:local-gaussian-approximation} hold, then
\begin{equation}
\sup_{x\ge0}|f_n(\tau_nx)-g_n(\tau_nx)|\longrightarrow0.
\label{eq:affine-local-comparison}
\end{equation}
\end{proposition}

\begin{proof}
The affine image of $X_n-\mu_n\sim N(0,\Sigma_n)$ is $N(0,a_n^{\mathsf T}\Sigma_na_n)=N(0,\tau_n^2)$, proving \eqref{eq:affine-image-normal}. Thus $Z=(\Theta_n(X_n)-\alpha_n)/\tau_n\sim N(0,1)$ and $g_n(\tau_nx)=\P(|Z|<x)=2\Phi(x)-1$; the final assertion is \eqref{eq:local-mass-uniform}.
\end{proof}

\section{Euclidean-ball profiles}
\label{sec:ball}

We apply the preceding results to parallel hyperplane sections of a Euclidean ball, the fixed-direction setting considered by Brehm and Voigt; see \refcitenote{Sect.~3}{BrehmVoigt2000}. Let \(n\ge2\), write \(N:=n-1\), and consider
\(B^n(\mu_n,R_n):=\{x\in\R^n:\|x-\mu_n\|\le R_n\}\), where \(R_n>0\). Fix a unit vector \(u_n\in\R^n\) and define
\begin{equation}
\Theta_n(x):=\frac{\pi}{2R_n}\langle u_n,x-\mu_n\rangle.
\label{eq:ball-observation-map}
\end{equation}
Then \(\nabla\Theta_n=(\pi/(2R_n))u_n\), so \(\|\nabla\Theta_n\|=\pi/(2R_n)\) and \(\Theta_n(\mu_n)=0\). Moreover,
\[
|\Theta_n(x)|\le \frac{\pi}{2R_n}\|x-\mu_n\|\le\frac\pi2
\quad (x\in B^n(\mu_n,R_n)),
\]
while for every \(\theta\in[-\pi/2,\pi/2]\), the point \(\mu_n+(2R_n/\pi)\theta u_n\) belongs to the ball and is mapped to \(\theta\). Hence
\begin{equation}
\Theta_n\bigl(B^n(\mu_n,R_n)\bigr)=I:=\left[-\frac\pi2,\frac\pi2\right].
\label{eq:ball-image-interval}
\end{equation}

\subsection{The section profile and fixed-scale concentration}

Fix \(\theta\in I\). The level equation \(\Theta_n(x)=\theta\) is equivalent to \(\langle u_n,x-\mu_n\rangle=(2R_n/\pi)\theta\). Thus every point on the level hyperplane can be written as
\(x-\mu_n=(2R_n/\pi)\theta u_n+z\) with \(z\in u_n^\perp\). Since \(u_n\perp z\),
\[
\|x-\mu_n\|^2
=\frac{4R_n^2}{\pi^2}\theta^2+\|z\|^2
\le R_n^2
\quad\Longleftrightarrow\quad
\|z\|\le R_n\sqrt{1-\frac{4\theta^2}{\pi^2}}.
\]
Consequently, this level section is an \((n-1)\)-ball of radius \(R_n\sqrt{1-4\theta^2/\pi^2}\). If \(\omega_{n-1}\) denotes the volume of the unit ball in \(\R^{n-1}\), then
\begin{align}
&\Vol_{n-1}\!\left(B^n(\mu_n,R_n)\cap\Theta_n^{-1}(\{\theta\})\right) \notag\\
&\qquad=\omega_{n-1}R_n^{n-1}
\left(1-\frac{4\theta^2}{\pi^2}\right)^{N/2}.
\label{eq:ball-section-volume}
\end{align}
Dividing by the central-section volume \(\omega_{n-1}R_n^{n-1}\) gives the normalized profile
\begin{equation}
\rho_n(\theta)=\left(1-\frac{4\theta^2}{\pi^2}\right)^{N/2},
\qquad |\theta|\le\frac\pi2.
\label{eq:ball-profile}
\end{equation}
It is positive on the interior of \(I\), vanishes at the endpoints, and has the unique maximum \(m_n=1\) at \(\alpha_n=0\). Hence \(A_n=\{0\}\).

For fixed \(0<\delta<\pi/2\) and \(c\in(0,1)\), the tail height is
\(\sup_{|\theta|\ge\delta}\rho_n(\theta)=(1-4\delta^2/\pi^2)^{N/2}\).
Moreover, \(\rho_n(\theta)\ge c\) is equivalent to
\(|\theta|\le(\pi/2)\sqrt{1-c^{2/N}}\); this interval lies in \((-\delta,\delta)\) for all sufficiently large \(N\), and hence
\(|L_{n,\delta}^{(c)}|=\pi\sqrt{1-c^{2/N}}\).
Since \(c^{2/N}=\exp(2\log c/N)=1+2\log c/N+O(N^{-2})\),
\begin{equation}
|L_{n,\delta}^{(c)}|
=\pi\sqrt{-\frac{2\log c}{N}}\left(1+O\!\left(\frac1N\right)\right).
\label{eq:ball-core-width-asymptotic}
\end{equation}
Set \(q_\delta:=1-4\delta^2/\pi^2\in(0,1)\). By \iref{prop:tail-to-core}, for some \(C_{\delta,c}>0\),
\begin{equation}
1-f_n(\delta)
\le C_{\delta,c}\sqrt N\,q_\delta^{N/2}
\longrightarrow0.
\label{eq:ball-fixed-scale-limit}
\end{equation}
If \(\delta\ge\pi/2\), then \(f_n(\delta)=1\): when
\(\delta=\pi/2\), the only excluded points are the endpoints, which have
\(\nu_n\)-measure zero. Therefore \(f_n(\delta)\to1\) for every fixed
\(\delta>0\).

\subsection{Normalization, moments, and the two natural scales}

The normalizing constant is
\(Z_n=\int_{-\pi/2}^{\pi/2}(1-4\theta^2/\pi^2)^{N/2}\dd\theta\).
With \(u=2\theta/\pi\) and then \(s=u^2\),
\begin{align}
Z_n
&=\frac\pi2\int_{-1}^{1}(1-u^2)^{N/2}\dd u
=\frac\pi2\int_0^1 s^{-1/2}(1-s)^{N/2}\dd s \notag\\
&=\frac\pi2\,\mathrm B\!\left(\frac12,\frac N2+1\right)
=\frac{\pi^{3/2}}2\frac{\Gamma(N/2+1)}{\Gamma(N/2+3/2)}.
\label{eq:ball-Zn-gamma}
\end{align}

\begin{proposition}
\label{prop:ball-exact-moments}
For every integer \(k\ge1\),
\begin{equation}
\E_{\nu_n}[\theta^{2k-1}]=0,
\qquad
\E_{\nu_n}[\theta^{2k}]
=\left(\frac\pi2\right)^{2k}
\frac{(2k-1)!!}{\prod_{j=1}^{k}(n+2j)}.
\label{eq:ball-even-moments}
\end{equation}
\end{proposition}

\begin{proof}
The odd moments vanish because \(\rho_n\) is even. For an even moment, the substitution \(u=2\theta/\pi\) gives
\[
\E_{\nu_n}[\theta^{2k}]
=\left(\frac\pi2\right)^{2k}
\frac{\int_0^1u^{2k}(1-u^2)^{N/2}\dd u}
{\int_0^1(1-u^2)^{N/2}\dd u}.
\]
Putting \(s=u^2\), the numerator and denominator become one half of
\(\mathrm B(k+1/2,N/2+1)\) and \(\mathrm B(1/2,N/2+1)\), respectively. Hence
\begin{align*}
\frac{\mathrm B(k+1/2,N/2+1)}{\mathrm B(1/2,N/2+1)}
&=\frac{\Gamma(k+1/2)}{\Gamma(1/2)}
\frac{\Gamma(N/2+3/2)}{\Gamma(N/2+k+3/2)}\\
&=\frac{(2k-1)!!}{2^k}
\frac{2^k}{\prod_{j=1}^{k}(N+2j+1)}
=\frac{(2k-1)!!}{\prod_{j=1}^{k}(n+2j)},
\end{align*}
which proves \eqref{eq:ball-even-moments}.
\end{proof}

In particular,
\begin{equation}
\Var_{\nu_n}(\theta)=\frac{\pi^2}{4(n+2)},
\qquad
\E_{\nu_n}[\theta^4]=\frac{3\pi^4}{16(n+2)(n+4)},
\label{eq:ball-variance-fourth}
\end{equation}
so the standardized fourth moment is \(3(n+2)/(n+4)\) and the excess kurtosis is
\begin{equation}
-\frac6{n+4}.
\label{eq:ball-excess-kurtosis}
\end{equation}

The logarithmic profile is defined only on \((-\pi/2,\pi/2)\). Set
\(V_n(\theta):=\log\rho_n(\theta)\). Thus
\(V_n(\theta)=\frac N2\log(1-4\theta^2/\pi^2)\), and direct differentiation gives
\begin{align*}
V_n'(\theta)&=-\frac{4N\theta}{\pi^2-4\theta^2},\\
V_n''(\theta)&=-\frac{4N(\pi^2+4\theta^2)}{(\pi^2-4\theta^2)^2},
\end{align*}
so \(V_n'(0)=0\) and \(-V_n''(0)=4N/\pi^2\). Therefore
\begin{equation}
\lambda_n=\frac{4N}{\pi^2},
\qquad
\tau_n=\lambda_n^{-1/2}=\frac\pi{2\sqrt N}=\frac\pi{2\sqrt{n-1}}.
\label{eq:ball-tau}
\end{equation}
The exact profile variance and the curvature variance are not identical:
\begin{equation}
\Var_{\nu_n}(\theta)=\frac{\pi^2}{4(n+2)},
\qquad
\tau_n^2=\frac{\pi^2}{4(n-1)},
\qquad
\frac{\Var_{\nu_n}(\theta)}{\tau_n^2}=\frac{n-1}{n+2}\to1.
\label{eq:ball-two-variances}
\end{equation}

This distinction persists in the physical axial coordinate \(s:=\langle u_n,x-\mu_n\rangle=(2R_n/\pi)\theta\). If \(X_n\) is uniformly distributed in \(B^n(\mu_n,R_n)\) and \(S_n:=\langle u_n,X_n-\mu_n\rangle\), then Fubini's theorem and \eqref{eq:ball-section-volume} give the density
\begin{equation}
q_n(s):=
\begin{cases}
\displaystyle
\frac{\omega_{n-1}(R_n^2-s^2)^{N/2}}{\omega_nR_n^n},
& |s|\le R_n,\\[5pt]
0, & |s|>R_n.
\end{cases}
\label{eq:ball-axial-marginal-density}
\end{equation}
Under the map \(s\mapsto\theta=\pi s/(2R_n)\), the probability measure
with density \(q_n\) is pushed forward to \(\nu_n\). This is the
one-dimensional instance of the marginal identity of Eldan and Klartag
\refcitenote{Eq.~(5)}{EldanKlartag2008}; taking
\(f=(\omega_nR_n^n)^{-1}\mathbf 1_{B^n(\mu_n,R_n)}\) and
\(E=\operatorname{span}\{u_n\}\) gives
\eqref{eq:ball-axial-marginal-density}. Therefore
\begin{equation}
\Var(S_n)
=\left(\frac{2R_n}{\pi}\right)^2\Var_{\nu_n}(\theta)
=\frac{R_n^2}{n+2}.
\label{eq:ball-physical-exact-variance}
\end{equation}
By contrast, the curvature scale in the physical coordinate is
\begin{equation}
\left(\frac{2R_n}{\pi}\right)\tau_n=\frac{R_n}{\sqrt{n-1}},
\qquad
\left(\frac{2R_n}{\pi}\right)^2\tau_n^2=\frac{R_n^2}{n-1}.
\label{eq:ball-physical-curvature-scale}
\end{equation}
Thus \(R_n^2/(n+2)\) is the exact marginal variance, whereas \(R_n^2/(n-1)\) is the variance of the curvature-matched Gaussian.

\begin{remark}
If $\theta\sim\nu_n$ and $W_n=1/2+\theta/\pi$, then $W_n\sim\operatorname{Beta}((N+2)/2,(N+2)/2)$ because $1-4\theta^2/\pi^2=4W_n(1-W_n)$. This is Ouimet's symmetric Dirichlet case \refcite{Ouimet2022}; its exact-variance and curvature calibrations agree to leading order but differ at order $N^{-1}$.
\end{remark}

\subsection{Curvature rescaling and the Gaussian limit}

For fixed $M>0$ and $|y|\le M$,
$\lambda_n^{-1}|V_n''(\tau_ny)-V_n''(0)|=|(1+y^2/N)/(1-y^2/N)^2-1|=O_M(N^{-1})$, so \eqref{eq:local-curvature-condition} holds. Also $\dist(0,\partial I)/\tau_n=\sqrt N\to\infty$, which gives asymptotic interiority.

Set \(y=\theta/\tau_n=(2\sqrt N/\pi)\theta\) and define the zero-extended rescaled profile
\begin{equation}
h_N(y):=
\begin{cases}
\displaystyle\left(1-\frac{y^2}{N}\right)^{N/2},
& |y|\le\sqrt N,\\[4pt]
0, & |y|>\sqrt N.
\end{cases}
\label{eq:ball-hN}
\end{equation}
For \(|y|<\sqrt N\), the inequality \(\log(1-t)\le-t\) gives
\[
\log h_N(y)=\frac N2\log\left(1-\frac{y^2}{N}\right)\le-\frac{y^2}{2},
\]
and therefore
\begin{equation}
0\le h_N(y)\le e^{-y^2/2}
\qquad (y\in\R).
\label{eq:ball-gaussian-domination}
\end{equation}
In particular,
\[
\sup_N\int_{|y|>M}h_N(y)\dd y
\le\int_{|y|>M}e^{-y^2/2}\dd y\longrightarrow0,
\]
so the rescaled tail-tightness condition is automatic.

Let \(D_N:=\int_\R h_N(y)\dd y\). The substitutions \(u=y/\sqrt N\) and \(s=u^2\) give
\begin{align}
D_N
&=\sqrt N\int_{-1}^{1}(1-u^2)^{N/2}\dd u
=\sqrt N\,\mathrm B\!\left(\frac12,\frac N2+1\right) \notag\\
&=\sqrt{\pi N}\frac{\Gamma(N/2+1)}{\Gamma(N/2+3/2)}.
\label{eq:ball-DN-gamma}
\end{align}
The normalized rescaled density is
\begin{equation}
p_n(y):=\frac{h_N(y)}{D_N}.
\label{eq:ball-pn}
\end{equation}

\begin{proposition}
\midlabel{proposition}{prop:ball-gaussian-limit}
As \(n\to\infty\),
\begin{equation}
\int_\R|p_n(y)-\varphi(y)|\dd y\longrightarrow0,
\qquad
d_{\rm TV}(p_n,\varphi)\longrightarrow0.
\label{eq:ball-pn-l1}
\end{equation}
\end{proposition}

\begin{proof}
The relative-curvature and interiority conditions were verified above, and \eqref{eq:ball-gaussian-domination} gives the required rescaled tail tightness. Thus \iref{thm:local-gaussian-approximation} applies to \eqref{eq:ball-pn}.
\end{proof}

Sodin's related spherical result gives quantitative Gaussian asymptotics on the growing regime $t=o(n^{1/4})$ \refcitenote{Lemma~1}{Sodin2007}. Here \iref{prop:ball-gaussian-limit} is global in $L^1$ at the curvature scale, and the expansion below identifies the $N^{-1}$ discrepancy.

The exact moments are also consistent with this Gaussian limit. From \(y=\theta/\tau_n\) and \eqref{eq:ball-even-moments},
\begin{equation}
\int_\R y^2p_n(y)\dd y=\frac{N}{N+3}\to1,
\qquad
\int_\R y^4p_n(y)\dd y=\frac{3N^2}{(N+3)(N+5)}\to3.
\label{eq:ball-rescaled-moments}
\end{equation}

\subsection{Gaussian calibration at the curvature scale}

For the affine observation \eqref{eq:ball-observation-map}, \(a_n:=\nabla\Theta_n=(\pi/(2R_n))u_n\) and \(e_n=u_n\). Set \(\varepsilon_n:=R_n^2/(n^2N)\). Substitution into \eqref{eq:directional-covariance} gives
\begin{equation}
\Sigma_n^{(L)}
=\frac{R_n^2}{N}u_nu_n^{\mathsf T}
+\frac{R_n^2}{n^2N}(I_n-u_nu_n^{\mathsf T}).
\label{eq:ball-covariance}
\end{equation}
By \iref{prop:directional-covariance}, this matrix is positive definite and
\begin{align}
a_n^{\mathsf T}\Sigma_n^{(L)}a_n
&=\tau_n^2=\frac{\pi^2}{4N}, \notag\\
\tr(\Sigma_n^{(L)})
&=\frac{R_n^2}{N}+\frac{R_n^2}{n^2}
=\frac{R_n^2}{N}\left(1+\frac{N}{n^2}\right).
\label{eq:ball-covariance-calibration}
\end{align}
Moreover, since \(K_n=\pi/(2R_n)\), \eqref{eq:ball-covariance-calibration} gives
\[
K_n^2\tr(\Sigma_n^{(L)})
=\frac{\pi^2}{4}\left(\frac1N+\frac1{n^2}\right)
\longrightarrow0.
\]
Thus the fixed-scale Gaussian concentration condition is satisfied independently of the growth of \(R_n\).

If $X_n\sim N(\mu_n,\Sigma_n^{(L)})$, then \Cref{prop:affine-local-calibration} gives $\Theta_n(X_n)\sim N(0,\tau_n^2)$ and hence
\begin{equation}
 g_n(\tau_nx)=2\Phi(x)-1,
 \qquad x\ge0.
 \label{eq:ball-gn-local-exact}
\end{equation}
Also, for every fixed $\delta>0$,
$g_n(\delta)=2\Phi(\delta/\tau_n)-1\to1$ because $\tau_n\to0$.
Combining this with \eqref{eq:ball-fixed-scale-limit} and using
$|f_n(\delta)-g_n(\delta)|\le |1-f_n(\delta)|+|1-g_n(\delta)|$, we obtain
$|f_n(\delta)-g_n(\delta)|\to0$.

\subsection{Exact local masses}

For $x\ge0$, write $F_n(x)=f_n(\tau_nx)=\int_{-x}^{x}p_n(y)\dd y$. For $0\le x\le\sqrt N$, the substitution $s=y^2/N$ gives
\begin{equation}
F_n(x)=I_{x^2/N}\!\left(\frac12,\frac N2+1\right),
\qquad 0\le x\le\sqrt N,
\label{eq:ball-exact-symmetric-mass}
\end{equation}
where $I_z(a,b)$ is the regularized incomplete beta function; for $x\ge\sqrt N$, $F_n(x)=1$.

\subsection{First-order density asymptotics}

The estimate \refcitenote{Theorem~3.1}{BrehmVoigt2000} of Brehm and
Voigt gives an \(O(N^{-1})\) upper bound for the one-dimensional
Euclidean-ball marginal in \(L^1\). To identify the leading density error,
we first establish an \(L^1\)-expansion of the unnormalized profile.

\begin{lemma}
\midlabel{lemma}{lem:ball-unnormalized-first-order}
As \(N\to\infty\),
\begin{equation}
N\bigl(h_N-e^{-y^2/2}\bigr)
\longrightarrow-\frac{y^4}{4}e^{-y^2/2}
\qquad\text{in }L^1(\R).
\label{eq:ball-unnormalized-first-order}
\end{equation}
\end{lemma}

\begin{proof}
Split \(\R\) into \(|y|\le\sqrt{N/2}\) and its complement. On the central region set \(s=y^2/N\in[0,1/2]\). Since
\[
-\log(1-s)=s+\frac{s^2}{2}+R(s),
\qquad
0\le R(s)=\sum_{k=3}^{\infty}\frac{s^k}{k}
\le\frac{s^3}{3(1-s)}\le\frac{2s^3}{3},
\]
we obtain
\begin{equation}
\frac N2\log(1-s)
=-\frac{y^2}{2}-\frac{y^4}{4N}-r_N(y),
\qquad
0\le r_N(y)\le\frac{y^6}{3N^2}.
\label{eq:ball-log-first-order}
\end{equation}
Hence
\(h_N(y)=e^{-y^2/2}\exp[-y^4/(4N)-r_N(y)]\). Put \(a=y^4/(4N)\). The elementary estimates \(|e^{-a-r}-e^{-a}|\le r\) and \(|e^{-a}-1+a|\le a^2/2\) for \(a,r\ge0\) give
\begin{align*}
&\left|N\bigl(h_N(y)-e^{-y^2/2}\bigr)+\frac{y^4}{4}e^{-y^2/2}\right|\\
&\qquad\le e^{-y^2/2}\left(Nr_N(y)+\frac{N a^2}{2}\right)
\le\frac1N\left(\frac{y^6}{3}+\frac{y^8}{32}\right)e^{-y^2/2}.
\end{align*}
After extending the central-region error by zero outside
\(\{|y|\le\sqrt{N/2}\}\), it is dominated by
\[
\left(\frac{y^6}{3}+\frac{y^8}{32}\right)e^{-y^2/2},
\]
and converges pointwise to zero. Hence dominated convergence shows that its
integral over the central region tends to zero.

On \(|y|>\sqrt{N/2}\), \eqref{eq:ball-gaussian-domination} gives
\[
\left|N(h_N-e^{-y^2/2})+\frac{y^4}{4}e^{-y^2/2}\right|
\le\left(2N+\frac{y^4}{4}\right)e^{-y^2/2}.
\]
Put \(R:=\sqrt{N/2}\). For \(R\ge1\), integration by parts gives
\begin{align*}
\int_R^\infty e^{-y^2/2}\dd y
&\le R^{-1}e^{-R^2/2},\\
\int_R^\infty y^4e^{-y^2/2}\dd y
&=R^3e^{-R^2/2}+3\int_R^\infty y^2e^{-y^2/2}\dd y\\
&\le (R^3+3R+3R^{-1})e^{-R^2/2}.
\end{align*}
Consequently,
\[
\int_{|y|>R}\left(2N+\frac{y^4}{4}\right)e^{-y^2/2}\dd y
\le C\,(NR^{-1}+R^3+R+R^{-1})e^{-R^2/2}\longrightarrow0.
\]
Together with the central-region estimate, this proves \eqref{eq:ball-unnormalized-first-order}.
\end{proof}

\begin{corollary}
\midlabel{corollary}{cor:ball-normalization-expansion}
As \(N\to\infty\),
\begin{equation}
D_N=\sqrt{2\pi}\left(1-\frac{3}{4N}+o(N^{-1})\right).
\label{eq:ball-DN-expansion}
\end{equation}
\end{corollary}

\begin{proof}
Integrating \eqref{eq:ball-unnormalized-first-order} and using
\(\int_\R y^4e^{-y^2/2}\dd y=3\sqrt{2\pi}\) gives
\[
N(D_N-\sqrt{2\pi})\longrightarrow-\frac34\sqrt{2\pi},
\]
which is equivalent to \eqref{eq:ball-DN-expansion}.
\end{proof}

Define
\begin{equation}
Q(y):=\frac{3-y^4}{4}\varphi(y).
\label{eq:ball-Q}
\end{equation}

\begin{theorem}
\midlabel{theorem}{thm:ball-first-order-density}
As \(n\to\infty\),
\begin{equation}
N(p_n-\varphi)\longrightarrow Q
\qquad\text{in }L^1(\R).
\label{eq:ball-first-order-density}
\end{equation}
\end{theorem}

\begin{proof}
Write \(e(y):=e^{-y^2/2}\). Since \(p_n=h_N/D_N\) and \(\varphi=e/\sqrt{2\pi}\),
\begin{equation}
N(p_n-\varphi)
=\frac{N(h_N-e)}{D_N}
+Ne\left(\frac1{D_N}-\frac1{\sqrt{2\pi}}\right).
\label{eq:ball-density-decomposition}
\end{equation}
The first term converges in \(L^1\) to \(-y^4\varphi(y)/4\) by \iref{lem:ball-unnormalized-first-order}. Moreover, \iref{cor:ball-normalization-expansion} gives
\[
\frac1{D_N}
=\frac1{\sqrt{2\pi}}\left(1+\frac{3}{4N}+o(N^{-1})\right),
\]
so the second term converges in \(L^1\) to \(3\varphi/4\). Adding the two limits in \eqref{eq:ball-density-decomposition} proves \eqref{eq:ball-first-order-density}.
\end{proof}

\begin{corollary}
\label{cor:ball-local-uniform-expansion}
For every fixed \(M>0\),
\begin{equation}
p_n(y)=\varphi(y)\left[1+\frac{3-y^4}{4N}+o_M(N^{-1})\right]
\qquad (|y|\le M),
\label{eq:ball-local-density-expansion}
\end{equation}
where the remainder is uniform on \([-M,M]\).
\end{corollary}

\begin{proof}
On $|y|\le M$, \eqref{eq:ball-log-first-order} and $e^{-t}=1-t+O(t^2)$ give, uniformly,
\begin{equation}
\frac{h_N(y)}{e^{-y^2/2}}
=1-\frac{y^4}{4N}+o_M(N^{-1})
\qquad (|y|\le M).
\label{eq:ball-hN-local-uniform}
\end{equation}
Moreover, \iref{cor:ball-normalization-expansion} gives $\sqrt{2\pi}/D_N=1+3/(4N)+o(N^{-1})$. Multiplying the two expansions proves \eqref{eq:ball-local-density-expansion}.
\end{proof}

\subsection{First-order comparison under exact-variance scaling}

For standardization by the exact standard deviation, \eqref{eq:ball-two-variances} gives
\begin{align}
 \sigma_n^2&:=\Var_{\nu_n}(\theta)=\frac{\pi^2}{4(N+3)},
 \qquad
 \kappa_n:=\frac{\sigma_n}{\tau_n}=\sqrt{\frac{N}{N+3}},
 \label{eq:ball-exact-variance-scale}\\
 \widehat p_n(z)&:=\kappa_n p_n(\kappa_n z),\qquad z\in\R.
 \label{eq:ball-exact-variance-density}
\end{align}
The latter is the density of $\theta/\sigma_n$ and has mean zero and variance one.

\begin{theorem}
\midlabel{theorem}{thm:ball-exact-variance-first-order}
As \(n\to\infty\),
\begin{equation}
 N(\widehat p_n-\varphi)
 \longrightarrow
 \widehat Q,
 \qquad
 \widehat Q(z):=-\frac14\bigl(z^4-6z^2+3\bigr)\varphi(z),
 \quad\text{in }L^1(\R).
 \label{eq:ball-exact-variance-first-order}
\end{equation}
\end{theorem}

\begin{proof}
Put \(G_n:=N(p_n-\varphi)\). By \iref{thm:ball-first-order-density},
\(G_n\to Q\) in \(L^1(\R)\), where \(Q(z)=(3-z^4)\varphi(z)/4\). From \eqref{eq:ball-exact-variance-density},
\begin{equation}
 N\bigl(\widehat p_n(z)-\varphi(z)\bigr)
 =\kappa_nG_n(\kappa_n z)
 +N\bigl[\kappa_n\varphi(\kappa_n z)-\varphi(z)\bigr].
 \label{eq:ball-exact-variance-decomposition}
\end{equation}
For the first term,
\begin{align*}
 \|\kappa_nG_n(\kappa_n\,\cdot)-Q\|_{L^1}
 &\le \|\kappa_n(G_n-Q)(\kappa_n\,\cdot)\|_{L^1}
 +\|\kappa_nQ(\kappa_n\,\cdot)-Q\|_{L^1}\\
 &=\|G_n-Q\|_{L^1}
 +\|\kappa_nQ(\kappa_n\,\cdot)-Q\|_{L^1}
 \longrightarrow0,
\end{align*}
because dilations are continuous in \(L^1(\R)\).

For the second term, define \(T_c(z):=c\varphi(cz)\). Since
\[
 \partial_cT_c(z)=(1-c^2z^2)\varphi(cz),
\]
the difference quotient \((T_c-T_1)/(c-1)\) converges in \(L^1(\R)\) to \((1-z^2)\varphi(z)\) as \(c\to1\). Indeed, for \(c\in[1/2,1]\), the derivative is bounded in absolute value by \(C(1+z^2)e^{-z^2/8}\), which is integrable, so the claim follows from the mean value theorem and dominated convergence.
Moreover,
\[
 N(\kappa_n-1)
 =N\left(\sqrt{\frac{N}{N+3}}-1\right)
 \longrightarrow-\frac32.
\]
Consequently,
\begin{equation}
 N\bigl[\kappa_n\varphi(\kappa_n\,\cdot)-\varphi\bigr]
 \longrightarrow-\frac32(1-z^2)\varphi(z)
 \quad\text{in }L^1(\R).
 \label{eq:ball-gaussian-rescaling-correction}
\end{equation}
Combining the two limits in \eqref{eq:ball-exact-variance-decomposition} gives
\[
 Q(z)-\frac32(1-z^2)\varphi(z)
 =-\frac14(z^4-6z^2+3)\varphi(z),
\]
which proves \eqref{eq:ball-exact-variance-first-order}.
\end{proof}

\begin{corollary}
\midlabel{corollary}{cor:ball-exact-variance-consequences}
Let \(\widehat F_n(x):=\int_{-x}^{x}\widehat p_n(z)\dd z\). For every fixed \(x\ge0\),
\begin{equation}
 N\left[\widehat F_n(x)-\bigl(2\Phi(x)-1\bigr)\right]
 \longrightarrow
 \widehat H(x):=\frac{x^3-3x}{2}\varphi(x).
 \label{eq:ball-exact-variance-local-mass}
\end{equation}
Furthermore, with
\begin{equation}
 b_-:=\sqrt{3-\sqrt6},
 \qquad
 b_+:=\sqrt{3+\sqrt6},
 \label{eq:ball-exact-variance-roots}
\end{equation}
one has
\begin{equation}
 N\,d_{\rm TV}(\widehat p_n,\varphi)
 \longrightarrow
 \widehat C_{\rm TV}
 :=\frac{\sqrt6}{2}
 \bigl[b_-\varphi(b_-)+b_+\varphi(b_+)\bigr]
 \approx0.350075038.
 \label{eq:ball-exact-variance-sharp-TV}
\end{equation}
\end{corollary}

\begin{proof}
The preceding \(L^1\)-convergence permits integration over \([-x,x]\) and also gives
\[
 N\,d_{\rm TV}(\widehat p_n,\varphi)
 \longrightarrow\frac12\int_\R|\widehat Q(z)|\dd z.
\]
Since
\begin{equation}
 \widehat Q(z)
 =\frac14\frac{\dd}{\dd z}\bigl[(z^3-3z)\varphi(z)\bigr],
 \label{eq:ball-exact-variance-antiderivative}
\end{equation}
integration over \([-x,x]\) proves \eqref{eq:ball-exact-variance-local-mass}. The polynomial \(z^4-6z^2+3\) vanishes at \(\pm b_-\) and \(\pm b_+\). Hence \(\widehat Q\) is positive exactly when \(b_-<|z|<b_+\). Since \(\int_\R\widehat Q(z)\dd z=0\),
\begin{align*}
 \frac12\int_\R|\widehat Q(z)|\dd z
 &=2\int_{b_-}^{b_+}\widehat Q(z)\dd z\\
 &=\frac12\bigl[(z^3-3z)\varphi(z)\bigr]_{b_-}^{b_+}\\
 &=\frac{\sqrt6}{2}
 \bigl[b_-\varphi(b_-)+b_+\varphi(b_+)\bigr],
\end{align*}
which is \eqref{eq:ball-exact-variance-sharp-TV}.
\end{proof}

Since $\widehat H$ has its minimum at $b_-$ and maximum at $b_+$, $\widehat C_{\rm TV}=\max_{x\ge0}\widehat H(x)-\min_{x\ge0}\widehat H(x)$, rather than $\max_{x\ge0}|\widehat H(x)|$.

\subsection{Local masses, total variation, and the density crossing}

\begin{theorem}
\midlabel{theorem}{thm:ball-local-mass-correction}
For every fixed \(x\ge0\),
\begin{equation}
N\left[F_n(x)-\bigl(2\Phi(x)-1\bigr)\right]
\longrightarrow H(x):=\frac{x^3+3x}{2}\varphi(x).
\label{eq:ball-local-mass-correction}
\end{equation}
\end{theorem}

\begin{proof}
By \iref{thm:ball-first-order-density}, the left-hand side converges to $\int_{-x}^{x}Q(y)\dd y$. Since $[(y^3+3y)\varphi(y)]'=(3-y^4)\varphi(y)$,
\[
\int_{-x}^{x}Q(y)\dd y
=\frac14\bigl[(y^3+3y)\varphi(y)\bigr]_{-x}^{x}
=\frac{x^3+3x}{2}\varphi(x).\qedhere
\]
\end{proof}

The derivative \(H'(x)=\frac12(3-x^4)\varphi(x)\) shows that \(H\) has its unique maximum at
\begin{equation}
a:=3^{1/4}.
\label{eq:ball-a}
\end{equation}

\begin{theorem}
\label{thm:ball-sharp-TV}
As \(n\to\infty\),
\begin{equation}
N\,d_{\rm TV}(p_n,\varphi)
\longrightarrow C_{\rm TV}:=\frac{a^3+3a}{2}\varphi(a),
\qquad a=3^{1/4}.
\label{eq:ball-sharp-TV}
\end{equation}
Numerically, \(C_{\rm TV}\approx0.522516163\).
\end{theorem}

\begin{proof}
\iref{thm:ball-first-order-density} implies
\[
N\,d_{\rm TV}(p_n,\varphi)
=\frac12\int_\R|N(p_n-\varphi)|\dd y
\longrightarrow\frac12\int_\R|Q(y)|\dd y.
\]
Now $\int_\R Q=0$, and $Q$ is positive exactly on $(-a,a)$. Hence
\[
 \frac12\int_\R|Q(y)|\dd y=\int_{-a}^{a}Q(y)\dd y
 =\frac14\bigl[(y^3+3y)\varphi(y)\bigr]_{-a}^{a}
 =\frac{a^3+3a}{2}\varphi(a).\qedhere
\]
\end{proof}

The density crossing can also be determined exactly. For \(0\le y<\sqrt N\), let \(R_N(y):=\log(p_n(y)/\varphi(y))\). Since
\[
R_N(y)=\frac N2\log\!\left(1-\frac{y^2}{N}\right)-\log D_N+\frac12\log(2\pi)+\frac{y^2}{2},
\]
direct differentiation gives
\begin{equation}
R_N'(y)=-\frac{y^3}{N-y^2}<0
\qquad (0<y<\sqrt N).
\label{eq:ball-density-ratio-derivative}
\end{equation}
Since \(\log(1-y^2/N)<-y^2/N\) for
\(0<|y|<\sqrt N\), we have \(D_N<\sqrt{2\pi}\), so \(R_N(0)>0\),
whereas \(R_N(y)\to-\infty\) as \(y\uparrow\sqrt N\). Thus there is a
unique \(r_n\in(0,\sqrt N)\) satisfying
\begin{equation}
p_n(r_n)=\varphi(r_n).
\label{eq:ball-crossing-equation}
\end{equation}
The sign of $p_n-\varphi$ is positive on $(-r_n,r_n)$ and negative outside. For any $\varepsilon\in(0,a)$, \eqref{eq:ball-local-density-expansion} gives $p_n(a-\varepsilon)>\varphi(a-\varepsilon)$ and $p_n(a+\varepsilon)<\varphi(a+\varepsilon)$ for large $n$. Hence $a-\varepsilon<r_n<a+\varepsilon$, so
\begin{equation}
r_n\longrightarrow3^{1/4}.
\label{eq:ball-crossing-limit}
\end{equation}
The zero integral and this sign pattern give the exact identity
\begin{equation}
d_{\rm TV}(p_n,\varphi)
=F_n(r_n)-\bigl(2\Phi(r_n)-1\bigr)
=I_{r_n^2/N}\!\left(\frac12,\frac N2+1\right)-\bigl(2\Phi(r_n)-1\bigr).
\label{eq:ball-exact-TV}
\end{equation}

Combining \eqref{eq:ball-gn-local-exact} with \iref{thm:ball-local-mass-correction} gives, for every fixed \(x\ge0\),
\begin{equation}
N\bigl[f_n(\tau_nx)-g_n(\tau_nx)\bigr]
\longrightarrow\frac{x^3+3x}{2}\varphi(x).
\label{eq:ball-profile-gaussian-first-order}
\end{equation}
Hence \(g_n(\tau_nx)\) is the leading term of the local mass, and \(H(x)/N\) is its first correction.








\begin{thebibliography}{99}

\bibitem{BrehmVoigt2000}
U.~Brehm and J.~Voigt,
\emph{Asymptotics of cross sections for convex bodies},
Beitr\"age Algebra Geom. \textbf{41} (2000), no.~2, 437--454.

\bibitem{DuembgenSamworthWellner2021}
L.~D\"umbgen, R.~J.~Samworth, and J.~A.~Wellner,
\emph{Bounding distributional errors via density ratios},
Bernoulli \textbf{27} (2021), no.~2, 818--852.
\newblock \href{https://doi.org/10.3150/20-BEJ1256}{doi:10.3150/20-BEJ1256}.

\bibitem{EldanKlartag2008}
R.~Eldan and B.~Klartag,
\emph{Pointwise estimates for marginals of convex bodies},
J. Funct. Anal. \textbf{254} (2008), no.~8, 2275--2293.

\bibitem{HenziDuembgenExtended2023}
A.~Henzi and L.~D\"umbgen,
\emph{Various new inequalities for beta distributions},
arXiv:2202.06718v9, 2023.

\bibitem{Ouimet2022}
F.~Ouimet,
\emph{A multivariate normal approximation for the Dirichlet density and some applications},
Stat \textbf{11} (2022), no.~1, Paper No.~e410, 13 pp.

\bibitem{Sodin2007}
S.~Sodin,
\emph{Tail-sensitive Gaussian asymptotics for marginals of concentrated measures in high dimension},
in \emph{Geometric Aspects of Functional Analysis: Israel Seminar 2004--2005},
Lecture Notes in Math., vol.~1910, Springer, Berlin, 2007, pp.~271--295.

\end{thebibliography}
\end{document}